\documentclass[11pt,a4paper]{article}
\usepackage{hyperref}
\usepackage[english]{babel}
\usepackage[utf8]{inputenc}
\usepackage[T1]{fontenc}
\usepackage{amsmath,amsthm,amssymb,amsfonts}
\numberwithin{equation}{section}
\usepackage{enumerate}
\usepackage{faktor}
\usepackage{dsfont}
\usepackage{scalerel,stackengine}
\usepackage{textcomp}
\usepackage{graphicx}
\usepackage{extarrows}
\usepackage{epigraph}
\usepackage{graphicx}
\usepackage{subcaption}
\usepackage{sidecap}
\usepackage{indentfirst}
\usepackage{tikz}
\usepackage{tikz-cd}
\usetikzlibrary{shapes.misc,arrows,matrix,patterns,decorations.markings,positioning}
\tikzset{cross/.style={cross out, draw=black, minimum size=2*(#1-\pgflinewidth), inner sep=0pt, outer sep=0pt},
cross/.default={4.5pt}}
\usepackage{pgfplots}
\usepackage{caption}
\usepackage{float}
\usepackage{color}
\usepackage{epstopdf}
\usepackage{epsfig}
\usepackage{stmaryrd}
\usepackage{color}
\usepackage{geometry}
\usepackage{multicol}
\usepackage{verbatim}

\DeclareMathOperator{\cchar}{char }
\DeclareMathOperator{\gal}{Gal}
\renewcommand{\geq}{\geqslant}
\renewcommand{\leq}{\leqslant} 
\renewcommand{\epsilon}{\varepsilon}

\newcommand{\Z}{\mathbb{Z}}
\newcommand{\Q}{\mathbb{Q}}
\newcommand{\F}{\mathbb{F}}

\newcommand{\K}{\mathbb K}

\newcommand{\LL}{\mathbb L}
\newcommand{\M}{\mathbb M}

\newcommand{\Gal}[2]{\gal\left(\faktor{#1}{#2}\right)}

\DeclareFontFamily{U}{mathx}{\hyphenchar\font45}
\DeclareFontShape{U}{mathx}{m}{n}{
      <5> <6> <7> <8> <9> <10>
      <10.95> <12> <14.4> <17.28> <20.74> <24.88>
      mathx10
      }{}
\DeclareSymbolFont{mathx}{U}{mathx}{m}{n}
\DeclareFontSubstitution{U}{mathx}{m}{n}
\DeclareMathAccent{\widecheck}{0}{mathx}{"71}
\DeclareMathAccent{\wideparen}{0}{mathx}{"75}

\newtheorem{teo}{Theorem}[section]
\newtheorem*{teo*}{Theorem}
\newtheorem{lemma}[teo]{Lemma}
\newtheorem{prop}[teo]{Proposition}
\newtheorem*{prop*}{Proposition}

\newtheorem{cor}[teo]{Corollary}

\usepackage{xpatch}
\makeatletter
\xpatchcmd{\@thm}{\thm@headpunct{.}}{\thm@headpunct{}}{}{}
\makeatother

\pgfplotsset{compat=1.13}
\begin{document}
\title{An elementary computation of the Galois groups of symmetric sextic trinomials}
\author{\scshape{Alberto Cavallo}\\ \\
 \footnotesize{Max Planck Institute for Mathematics,} \\
 \footnotesize{Bonn 53111, Germany}\\ \\
\small{cavallo@mpim-bonn.mpg.de}}
\date{}

\maketitle             

\begin{abstract}
 We compute the Galois group of the splitting field $\F$ of any irreducible and separable polynomial 
 $f(x)=x^6+ax^3+b$ with $a,b\in\K$, a field with characteristic different 
 from two.
 The proofs require to distinguish between two cases: whether or not the
 cubic roots of unity belong to $\K$.
 
 We also give a criterion to determine whether a polynomial as $f(x)$ is irreducible, when $\F$ is a finite field. Moreover, at the end of the paper we also give a complete list of all the possible subfields of $\F$.
\end{abstract}
\section{Introduction}
The computation of the Galois group of the splitting field of polynomials as $x^4+cx^2+d$ is a standard exercise in Galois theory, see \cite{Cox}. On the other hand, the solution of the similar problem for sextic trinomials $f(x)=x^6+ax^3+b$, although it may be known to experts (\cite{1,2,3,4,5,6,7}), cannot be found easily in literature; specially in the general case when the coefficients of $f(x)$ are taken in a generic field $\K$. This was our main motivation for writing this paper. 

Here, we determine the Galois group of $\F$, the splitting field of $f(x)$, whenever $f(x)\in\K[x]$ is irreducible and separable over $\K$, where $\K$ is every possible field with characteristic different from two. More specifically, we prove the following theorem. Note that when we write that the element $\sqrt[3]b$ is not in the field $\K$ we mean that none of the three cubic roots of $b$ belongs to $\K$. 
\begin{teo}
 \label{teo:main}
 Suppose that $f(x)=x^6+ax^3+b$ is an irreducible and separable polynomial over the field $\K$, where $\cchar(\K)\neq2$; moreover, consider $\Delta=a^2-4b$ and denote with $\F$ the splitting field of $f(x)$. Then we have that 
 \begin{enumerate}
     \item when $\zeta_3\notin\K$, where $\zeta_3$ is a primitive cubic root of unity, and $-3\Delta$        is not a square in $\K$:
     \begin{itemize}
         \item if $\sqrt[3]b\in\K$ or $R(x)=x^3-3bx+ab$ is reducible over $\K$ then $\Gal{\F}{\K}\cong           D_6$, the Dihedral group of order $12$;
         \item if $\sqrt[3]b\notin\K$ and $R(x)$ is irreducible over $\K$ then $\Gal{\F}{\K}\cong\mathcal       S_3\times\mathcal S_3$, the direct product of two symmetric groups over $3$ nodes.
     \end{itemize}
     \item When $\zeta_3\notin\K$ and $-3\Delta$ is a square in $\K$:
     \begin{itemize}
         \item if $\sqrt[3]b\in\K$ and $R(x)$ is irreducible over $\K$ then $\Gal{\F}{\K}\cong C_6$,  
               the cyclic group of order $6$;
         \item if $\sqrt[3]b\notin\K$ and $R(x)$ is reducible over $\K$ then $\Gal{\F}{\K}\cong\mathcal         S_3$;
         \item if $\sqrt[3]b\notin\K$ and $R(x)$ is irreducible over $\K$ then $\Gal{\F}{\K}\cong               C_3\times\mathcal S_3$.
     \end{itemize}
     \item When $\zeta_3\in\K$:
     \begin{itemize}
         \item if $\sqrt[3]b\in\K$ and $R(x)$ is irreducible over $\K$ then $\Gal{\F}{\K}\cong\mathcal          S_3$;
         \item if $\sqrt[3]b\notin\K$ and $R(x)$ is reducible over $\K$ then $\Gal{\F}{\K}\cong C_6$;
         \item if $\sqrt[3]b\notin\K$ and $R(x)$ is irreducible over $\K$ then $\Gal{\F}{\K}\cong               C_3\times\mathcal S_3$.
     \end{itemize}
 \end{enumerate}
 These are all the possible cases. 
\end{teo}
Such computations agree with the ones of Harrington and Jones in \cite{H-J} in the case $\K=\Q$ and $b$ is a cube in $\Q$.

We present an infinite family of polynomials whose Galois group is $G$ for every $G$ appearing in Theorem \ref{teo:main}. 
Using Galois correspondence, we are also able to explicitly determine all the possible extensions of $\K$ which are subfields of $\F$. Furthermore, we apply the results in Theorem \ref{teo:main} to finite fields;
we obtain necessary and sufficient conditions for a polynomial $f(x)$ of the given form to be irreducible.
\begin{teo}
 \label{teo:finite}
 Suppose that $f(x)=x^6+ax^3+b\in\F_{p^k}[x]$, where $p\neq2$. Then we have that
 \begin{enumerate} 
     \item when $p^k\not\equiv1\emph{ mod }3$ and $p\neq3$ the       polynomial $f(x)$ is irreducible if and only if 
           $R(x)=x^3-3bx+ab$ is irreducible over $\F_{p^k}$;
     \item when $p^k\equiv1\emph{ mod }3$ the polynomial 
           $f(x)$ is irreducible if and only if $\sqrt[3]b\:,\sqrt\Delta\notin\F_{p^k}$;
     \item when $p=3$ the polynomial $f(x)$ is never irreducible.
 \end{enumerate}
\end{teo}
We apply this result to prove that some polynomials are irreducible over the rational numbers, see Corollaries \ref{cor:1} and \ref{cor:2}.

The paper is organized as follows. In Section \ref{section:preliminaries} we set the notation and we prove some preliminary results. Moreover, we determine what are the possible degrees $[\F:\K]$ for the splitting field of $f(x)$ over $\K$.
In Section \ref{section:group} we prove Theorem \ref{teo:main} and finally, in Section \ref{section:subfields}, we list all the intermediate extensions between $\F$ and $\K$.
\paragraph*{Acknowledgements} The author has a post-doctoral fellowship at the Max Planck Institute for Mathematics in Bonn.

\section{Preliminaries}
\label{section:preliminaries}
Let us consider our polynomial $f(x)=x^6+ax^3+b$ over the field $\K$. We suppose that $f(x)$ is irreducible, and also separable if $\cchar(\K)=3$, and the characteristic of $\K$ is not two. We also take $\Delta=a^2-4b$.

Since one has \[x^3=\dfrac{-a\pm\sqrt\Delta}{2}\:,\]
the six roots of $f(x)$ in the algebraic closure of $\K$, when $\cchar(\K)$ is different from three, are $\left\{\alpha,\alpha\zeta_3,\alpha\zeta_3^2,\beta,\beta\zeta_3,\beta\zeta_3^2\right\}$, 
where $\zeta_3$ is a primitive cubic root of unity; moreover, we have that
\[\alpha=\sqrt[3]{\dfrac{-a+\sqrt\Delta}{2}}\:\:\:\:\:\text{ and }\:\:\:\:\:\beta=\sqrt[3]{\dfrac{-a-\sqrt\Delta}{2}}\:.\]
Otherwise, in the case of $\cchar(\K)=3$ the roots are just $\{\alpha,\beta\}$,
each one with multiplicity three. This is because in this kind of fields $\zeta_3$ does not exist and the only cubic root of unity is $1$.

We observe that $\alpha\beta=\sqrt[3]b$. Since once we fix $\alpha$ there are 
three choices for $\beta$, we have that $\sqrt[3]b$ also depends on this selection. Although, when $\K=\Q$, it may seems more natural to take $\alpha$ and $\beta$ in the way that $\sqrt[3]b$ is always real, in this paper we need to not make this assumption.

Furthermore, after an easy computation we note that
\[\alpha^3+\beta^3=-a\:\:\:\:\:\text{ and }\:\:\:\:\:\alpha^3\beta^3=b\:.\]
\begin{lemma}
 \label{lemma:triplets}
 We have that $\alpha\neq\beta$, where here we mean that none of the three values of $\beta$ coincides with $\alpha$.
\end{lemma} 
\begin{proof}
 Suppose for now that $\cchar(\K)\neq3$ and assume that $\alpha=\beta$. Since there is a double root one has \[\text{disc}(f(x))=729b^2\Delta^3=0\] and this implies that $\Delta$ or $b$ is also zero. In both cases $f(x)$ would be reducible.
 
 Now we consider a field $\K$ with characteristic three. In this case $\text{disc}(f(x))$ is always zero, but from the equations we wrote before we obtain
 \[(2\alpha)^3=(\alpha+\alpha)^3=2\alpha^3=2a\:,\] which means that
 \[\alpha^3=a\:\:\:\:\:\text{ and }\:\:\:\:\:\alpha^6=b=a^2\:.\]
 Hence, in conclusion $f(x)=x^6-2ax^3+a^2=(x^3-a)^2$ and this is a contradiction because $f(x)$ is irreducible.
\end{proof}
The splitting field of $f(x)$ is clearly $\F=\K(\zeta_3,\alpha,\beta)=\K(\zeta_3,\alpha,\sqrt[3]b)$. Denote with $\LL$ the subfield $\K(\zeta_3,\sqrt[3]b)\cap\K(\alpha)$;
we obtain the diagram of extensions in Figure \ref{Splitting}.
In the diagrams on the right we have that $[\K(\sqrt[3]b):\K]=1,3$ because, up to the choice of $\alpha$ and $\beta$, we can assume either $\sqrt[3]b\in\K$ or none of the three possible cubic roots of $b$ belongs to $\K$.
We can detect whether $\zeta_3$ belongs to $\K(\alpha)$ by using the following criterion.
\begin{lemma}
 \label{lemma:zeta}
 Suppose that $\K$ is a field with $\cchar(\K)\neq2,3$ and $\zeta_3\notin\K$. We have $\K(\zeta_3)\subset\K(\alpha)$ if and only if $\Delta=-3n^2$, where $n\in\K$.
\end{lemma}
\begin{figure}[H]
  \begin{center}
   \begin{tikzpicture}[node distance=2cm, auto]
               \node (A) {$\F$};
		       \node (B) [below of=A,xshift=2cm] {$\K(\alpha)$};
               \node (C) [below of=A,xshift=-2cm] {$\K(\zeta_3,\sqrt[3]b)$};
               \node (D) [below of=C,xshift=2cm]          {$\LL$};
               \node (E) [below of=D] {$\K$};
               \draw[-] (C) to node  {} (A);
               \draw[-] (B) to node  {} (A);
               \draw[-] (D) to node  [swap]{$2\text{ or }6$} (C);
               \draw[-] (D) to node  {} (B);
               \draw[-] (E) to node  {} (D);
               \draw[-] (E) to node  {$2\text{ or }6$} (C);
               \draw[-] (E) to node  [swap]{$6$} (B);
               
               \node (A1) [right of=B,xshift=0.8cm]{$\K(\alpha)$};
		       \node (B1) [below of=A1,xshift=1.5cm]   
		                 {$\K(\sqrt[3]b)$};
               \node (C1) [below of=A1,xshift=-1.5cm] {$\K(\sqrt\Delta)$};
               \node (D1) [below of=C1,xshift=1.5cm]{$\K$};
               \draw[-] (C1) to node  {$3$} (A1);
               \draw[-] (B1) to node  [swap]{$6,2$} (A1);
               \draw[-] (D1) to node  {$2$} (C1);
               \draw[-] (D1) to node  [swap]{$1,3$} (B1);
               
               \node (F) [right of=A,xshift=8cm] {$\F$};
		       \node (G) [below of=F,xshift=2cm] {$\K(\alpha)$};
               \node (H) [below of=F,xshift=-2cm] {$\K(\zeta_3,\sqrt[3]b)$};
               \node (I) [below of=H,xshift=2cm]          {$\K(\zeta_3)$};
               \node (L) [below of=I] {$\K$};
               \draw[-] (H) to node  {} (F);
               \draw[-] (G) to node  {} (F);
               \draw[-] (I) to node  {$3$} (H);
               \draw[-] (I) to node  {$6,3$} (G);
               \draw[-] (L) to node  {$1,2$} (I);
               \draw[-] (L) to node  [swap]{$6$} (G);
  \end{tikzpicture}
 \end{center}
 \caption{In the diagram on the left $\K(\zeta_3)\not\subset\K(\alpha)$, in the middle one $\K(\zeta_3),\:\K(\sqrt[3]b)\subset\K(\alpha)$, while in the one on the right $\K(\zeta_3)\subset\K(\alpha)$ and $\K(\sqrt[3]b)\not\subset\K(\alpha)$.}
 \label{Splitting}
\end{figure}
\begin{proof}
 The if implication follows by observing that \[2\alpha^3+a=\sqrt\Delta=\sqrt{-3n^2}=ni\sqrt3\] and then $\K(\zeta_3)=\K(i\sqrt3)$ is contained in $\K(\alpha)$.
 
 For the only if implication suppose that $\Delta\neq-3n^2$. Then $\K(\sqrt\Delta,i\sqrt3)\subset\K(\alpha)$ should be a degree 4 extension of $\K$, but
 obviously 4 is not a divisor of 6.
\end{proof}
When $\K(\zeta_3)\subset\K(\alpha)$ the diagram on the right in Figure \ref{Splitting} tells us that $[\F:\K]$ is $6$ or $18$. Since the only transitive subgroup of $\mathcal S_6$ of order $18$ is $C_3\times\mathcal S_3$
\cite{Transitive},
we obtain that $\Gal{\F}{\K}$ can be isomorphic to $\mathcal S_3,C_6$ or $C_3\times\mathcal S_3$; we recall that $C_n$
denotes the cyclic group of order $n$.

On the other hand, if $\zeta_3$ is not in $\K(\alpha)$ then, since $\K(\zeta_3)$ is the only quadratic extension in $\K(\zeta_3,\sqrt[3]b)$, we have that
$[\LL:\K]$ is equal to $1$ or $3$. Hence, the diagram on the left in Figure \ref{Splitting} implies $[\F:\K]=12,36$.

Now if the degree of $\F$ is $12$ then $\Gal{\F}{\K}\cong D_6$. In fact, the only transitive subgroups of $\mathcal S_6$ of order twelve are $D_6$ and $\mathcal A_4$, see \cite{Transitive}; the latter group does not have subgroups of order six and, under the Galois correspondence, this results in $\F$ not having a subfield of degree $2$. This is not the case since $\K(\zeta_3)\subset\F$.
\begin{lemma}
 Suppose that $[\F:\K]=36$ where the fields $\F$ and $\K$ are as before. Then we have that $\Gal{\F}{\K}\cong\mathcal S_3\times\mathcal S_3$.
\end{lemma}
\begin{proof}
 First we note that $\F$ has $\K(\zeta_3)$ and $\K(\sqrt\Delta)$ as  distinct normal extension of degree $2$ because $\K(\zeta_3)\not\subset\K(\alpha)$,
 see Figure \ref{Splitting}. 
 We can already conclude that $\Gal{\F}{\K}\cong\mathcal S_3\times\mathcal S_3$;
 in fact, the only other transitive subgroup of $\mathcal S_6$ of order $36$, up to isomorphism, is $\left(C_3\times C_3\right)\rtimes_{\phi}C_4$, where the action $\phi$ is faithful \cite{Transitive}. This group has a unique subgroup
 of order $18$, while by Galois correspondence it should have at least two of them. 
\end{proof}
Note that $[\F:\K]=36$ implies $\K(\sqrt[3]b)\not\subset\K(\alpha)$, which also means that $\K(\alpha)\neq\K(\beta)$. Therefore, the subfield $\K(\alpha)$ is not a normal extension of $\K$ and then the isomorphism between $\Gal{\F}{\K}$
and $\mathcal S_3\times\mathcal S_3$ is not given by the diagram on the left in Figure \ref{Splitting}; in the sense that the subfield that corresponds to $\{0\}\times\mathcal S_3$ is not $\K(\alpha)$.

We conclude this section with the following useful lemma. As remarked before, we say that $\sqrt[3]b$ belongs to the field $\K$ if at least one of the three cubic roots of $b$ is
in $\K$.
\begin{lemma}
 \label{lemma:both}
 Consider $f(x)=x^6+ax^3+b$ irreducible over a field $\K$ with $\cchar(\K)\neq2$.
 Then we cannot have both $\sqrt[3]b\in\K$ and $R(x)=x^3-3bx+ab$ reducible over $\K$.
\end{lemma}
\begin{proof}
 We can easily check that the three roots of $R(x)$, in the algebraic closure of $\K$, are $\{\alpha\beta(\alpha+\beta),\alpha\beta(\alpha\zeta_3+\beta\zeta_3^2),\alpha\beta(\alpha\zeta_3^2+\beta\zeta_3)\}$. Note that these roots are all distinct if $\cchar(\K)\neq3$.
 
 We can choose $\alpha$ and $\beta$ in the way that $\alpha\beta(\alpha+\beta)\in\K$ and there is an $i$ such that $\sqrt[3]b=\alpha\beta\zeta_3^i\in\K$. This tells us that $\alpha$ is the root of a degree two polynomial over $\K(\zeta_3)$.
 
 Since $[\K(\zeta_3):\K]\leq2$ we have that $[\K(\alpha):\K]\leq4$, but this is a contradiction because $\alpha$ is the root of an irreducible polynomial of degree six over $\K$.
\end{proof}

\section{The Galois groups}
\label{section:group}
\subsection{Proof of Theorem \ref{teo:main}: cubic roots of unity not in \texorpdfstring{$\K$}{K}}
\label{subsection:not}
We recall that $\cchar(\K)$ still cannot be equal to 2; and it is necessarily different from 3. Suppose for now that $\Delta\neq-3n^2$ for every $n\in\K$, which means that $\K(\zeta_3)\not\subset\K(\alpha)$ from Lemma \ref{lemma:zeta}.
\begin{prop}
 \label{prop:cube}
 If $\sqrt[3]b\in\K$ and $\Delta\neq-3n^2$ then $\Gal{\F}{\K}\cong D_6$.
\end{prop}
\begin{proof}
 We have already seen in the previous section that, in this case, the order of the Galois group can be $12$ or $36$; depending on the degree of $\LL=[\K(\zeta_3,\sqrt[3]b)\cap\K(\alpha):\K]$. Hence, since $\sqrt[3]b\in\K$  we have that $[\LL:\K]=1$ and this implies the thesis; in fact, we saw that $D_6$ is the only option when $[\F:\K]=12$. 
\end{proof}
From now on, we also suppose that $\sqrt[3]b\notin\K$. Moreover, we say that
$\Gal{\F}{\K}\cong D_6$ and study what happens to the coefficients $a$ and $b$.

First, we use Lemma \ref{lemma:triplets} to ensure that $\alpha\neq\beta$.
Since $[\F:\K]=12$ there is a cubic subfield of $\K(\zeta_3,\sqrt[3]b)$ inside $\K(\alpha)$. We take $\beta$ such that \[\sqrt[3]b=\alpha\beta\in\K(\alpha)=\K(\beta)\:.\]
We can describe all twelve automorphisms in the Galois group. The set \[\mathcal B=\left\{\alpha^i\zeta_3^j\:|\:i=0,...,5;\:j=1,2\right\}\] is a basis of $\F$ as a $\K$-vector space. Then $F\in\Gal{\F}{\K}$ is determined by:
\[\begin{aligned}
&F(\alpha)=\alpha\zeta_3^i\:\text{ or }\:\beta\zeta_3^k\:\:\:\:\:\hspace{2cm}i=0,1,2;\:\:\:k=0,1,2;\\
&F(\zeta_3)=\zeta_3^j\:\:\:\:\:\hspace{5cm}j=1,2\:.
\end{aligned}\]
These are all the possibilities, because $\alpha$ is a root of $f(x)$ and then $F(\alpha)$
needs to be a root too; moreover, clearly a primitive root of unity will be send to another primitive root of the same order.
We need to compute $F(\beta)$ and $F(\sqrt[3]b)$ for every $F$ in the Galois group. 

We denote the automorphisms with the following notation: we call $F_{(i,j)}$ the map that sends $\zeta_3$ to $\zeta_3^j$ for $j=1,2$, $\alpha$ to $\alpha\zeta_3^i$ for $i=0,1,2$ and $\alpha$ to $\beta\zeta_3^{i}$ for $i=3,4,5$. 

Tha mapo $F_{(0,1)}$ is the Identity and then this case is easy. The subfield $\K(\zeta_3)$ is a normal quadratic extension of $\K$. Therefore, it is fixed by the characteristic subgroup of order six of $D_6$. This is because $\K(\zeta_3)\subset\K(\zeta_3,\sqrt[3]b)$, which is a normal extension of $\K$ of degree six, and there is only one subgroup of $D_6$ of order six containing 
a normal subgroup of order two.

The previous argument tells us that $\K(\zeta_3)$ is fixed by the automorphisms $F_{(i,1)}$ for $i=0,...,5$ and then $F_{(i,2)}$ are the six symmetries of $D_6$.

Since $[\K(\alpha):\K]=6$ we have that $F_{(0,2)}$ fixes $\K(\alpha)$. Furthermore, the fact that $\K(\alpha)=\K(\beta)$ implies that $F_{(0,2)}$ also fixes $\beta$, and consequently $\sqrt[3]b$.
In the same way, the map $F_{(3,2)}$ is a symmetry and then it has order two.
This means that $F_{(3,2)}(\beta)=\alpha$. We can now prove the following lemma.
\begin{lemma}
 \label{lemma:subfield}
 If $\alpha$ and $\beta$ are taken as before then $\K(\alpha+\beta)\subset\K(\sqrt[3]b)$.
\end{lemma}
\begin{proof}
 We do some computations:
 \[(\alpha+\beta)^3=\alpha^3+\beta^3+3\alpha\beta(\alpha+\beta)=-a+3\sqrt[3]b(\alpha+\beta)\] and then $\alpha+\beta$ is a root of $g(x)=x^3-3\sqrt[3]b\:x+a\in\K(\sqrt[3]b)[x]$.
 
 The polynomial $g(x)$ does not have all three roots in $\K(\sqrt[3]b)$ because \[\sqrt{\text{disc(g(x))}}=\sqrt{-27\Delta}=3\cdot i\sqrt3\cdot\sqrt\Delta\notin\K(\sqrt[3]b)\:.\]
 Furthermore, $g(x)$ is not irreducible because otherwise $9$ would divide $6$.
 
 At the end, we have that either $\K(\alpha+\beta)\subset\K(\sqrt[3]b)$ or $[\K(\alpha+\beta):\K]=6$ and $\K(\alpha+\beta)=\K(\alpha)$. In the latter case $\alpha+\beta$ should only be fixed by $F_{(0,2)}$; this is not true because $F_{(3,2)}(\alpha+\beta)=\alpha+\beta$. 
\end{proof}
Now we can study the other automorphisms. The maps $F_{(1,1)}$ and $F_{(2,1)}$ are in the cyclic subgroup $H$ of order six of $D_6$. Since $F_{(1,1)}^3(\alpha)=F_{(2,1)}^3(\alpha)=\alpha$, we have that they are the elements of order three of the Galois group. Let us compute $F_{(1,1)}(\beta)$: it is easy to check that $\alpha\neq\beta$ implies $F_{(1,1)}(\beta)=\beta\zeta_3^j$, where $j$ is not zero because $F_{(1,1)}$ cannot fix $\beta$. Moreover, if $j=2$ then $F_{(1,1)}$
would fix $\sqrt[3]b$; this is impossible since $\K(\sqrt[3]b)$ is a cubic extension of $\K$ and then it is fixed by an order four subgroup of $D_6$. We conclude that $F_{(1,1)}(\beta)=\beta\zeta_3$ and, using the same proof, $F_{(2,1)}(\beta)=\beta\zeta_3^2$.

We now consider the other symmetries. Since their order is two we immediately obtain that $F_{(i,2)}(\beta)=\alpha\zeta_3^{i}$ for $i=4,5$. On the other hand, as before we have that $F_{(1,2)}(\beta)=\beta\zeta_3^j$ and $j$ is not zero. Again, if $j=2$ then $F_{(1,2)}$
would fix $\sqrt[3]b$; this is still impossible since $F_{(0,2)}$ and $F_{(3,2)}$ are the only two symmetries that could fix $\sqrt[3]b$.
We have gotten that $F_{(1,2)}(\beta)=\beta\zeta_3$ and again the same proof also gives $F_{(2,2)}(\beta)=\beta\zeta_3^2$.

We have three automorphisms left to consider: $F_{(3,1)},F_{(4,1)}$ and $F_{(5,1)}$. For Lemma \ref{lemma:subfield} the order of $F_{(3,1)}$ is two; in fact, it says that the order two element in $H$ has to fix $\alpha+\beta$. It is an easy check that this cannot happen if such an element is $F_{(4,1)}$ or $F_{(5,1)}$. In conclusion, one has $F_{(3,1)}(\beta)=\alpha$.

The maps $F_{(4,1)}$ and $F_{(5,1)}$ are the two elements of order six in $D_6$. We know that these two maps send $\beta$ to $\alpha\zeta_3^i$ for some $i$ because one has $\alpha\neq\beta$; in order to determine $i$ we use the following relation in $D_6$: \[r^k=sr^{6-k}s\:\:\:\:\:\text{for }k=0,...,5\:,\] where $s$ is a symmetry and $r$ is an order six rotation.
We use this equation with $k=5,r=F_{(4,1)}$ and $s=F_{(3,2)}$ and we get
\[\beta\zeta_3^2=F_{(5,1)}(\alpha)=(F_{(3,2)}\circ F_{(4,1)}\circ F_{(3,2)})(\alpha)=\beta\zeta_3^{2i}\:,\] which means that $F_{(4,1)}(\beta)=\alpha\zeta_3^i=\alpha\zeta_3$, and
\[F_{(5,1)}(\beta)=(F_{(3,2)}\circ F_{(4,1)}\circ F_{(3,2)})(\beta)=\alpha\zeta_3^{2}\:.\]
This concludes the study of the automorphisms in $\Gal{\F}{\K}$. We summarize the results in the following table.
\begin{figure}[H]
\begin{center}
 \renewcommand{\arraystretch}{2}
  \begin{tabular}{ | c | c || c | c | }
    \hline
    $F_{(0,1)}(\beta)=\beta$ & $F_{(0,1)}(\sqrt[3]b)=\sqrt[3]b$ & $F_{(0,2)}(\beta)=\beta$ & $F_{(0,2)}(\sqrt[3]b)=\sqrt[3]b$ \\ \hline
    $F_{(1,1)}(\beta)=\beta\zeta_3$ & $F_{(1,1)}(\sqrt[3]b)=\sqrt[3]b\:\zeta_3^2$ & $F_{(1,2)}(\beta)=\beta\zeta_3$ & $F_{(1,2)}(\sqrt[3]b)=\sqrt[3]b\:\zeta_3^2$ \\ \hline
    $F_{(2,1)}(\beta)=\beta\zeta_3^2$ & $F_{(2,1)}(\sqrt[3]b)=\sqrt[3]b\:\zeta_3$ & $F_{(2,2)}(\beta)=\beta\zeta_3^2$ & $F_{(2,2)}(\sqrt[3]b)=\sqrt[3]b\:\zeta_3$ \\ \hline 
    $F_{(3,1)}(\beta)=\alpha$ & $F_{(3,1)}(\sqrt[3]b)=\sqrt[3]b$ &  $F_{(3,2)}(\beta)=\alpha$ & $F_{(3,2)}(\sqrt[3]b)=\sqrt[3]b$ \\ \hline
    $F_{(4,1)}(\beta)=\alpha\zeta_3$ & $F_{(4,1)}(\sqrt[3]b)=\sqrt[3]b\:\zeta_3^2$ & $F_{(4,2)}(\beta)=\alpha\zeta_3$ & $F_{(4,2)}(\sqrt[3]b)=\sqrt[3]b\:\zeta_3^2$ \\ \hline
    $F_{(5,1)}(\beta)=\alpha\zeta_3^2$ & $F_{(5,1)}(\sqrt[3]b)=\sqrt[3]b\:\zeta_3$ & $F_{(5,2)}(\beta)=\alpha\zeta_3^2$ & $F_{(5,2)}(\sqrt[3]b)=\sqrt[3]b\:\zeta_3$ \\
    \hline
  \end{tabular}
 \renewcommand{\arraystretch}{1}
\end{center}
\caption{The values of $\beta$ and $\sqrt[3]b$ for the twelve                      automorphisms of $\Gal{\F}{\K}\cong D_6$.}
\label{D_6}
\end{figure} 
We have completely described the Galois group when $\sqrt[3]b\notin\K$ and $[\F:\K]=12$. This allows us to determine when $\Gal{\F}{\K}$ is isomorphic to the Dihedral group $D_6$.
\begin{prop}
 \label{prop:D_6}
 Take a field $\K$, not containing all the cubic roots of unity, such that $\cchar(\K)\neq2$ and an irreducible polynomial $f(x)=x^6+ax^3+b$ over $\K$. Suppose that $\sqrt[3]b\notin\K$ and $\Delta\neq-3n^2$ for every $n\in\K$.
 
 Then $\Gal{\F}{\K}\cong D_6$ if and only if $R(x)=x^3-3bx+ab$ is reducible over $\K$, where $\F$ is the splitting field of $f(x)$.
\end{prop}
\begin{proof}
 We start with the only if implication. We have that $\alpha\beta(\alpha+\beta)=c\in\K$ because is fixed by all the automorphisms in $\Gal{\F}{\K}$, see Figure \ref{D_6}.
 
 Then one has \[c^2=\alpha^2\beta^2(\alpha+\beta)^2=\alpha^2\beta^2(\alpha^2+\beta^2)+2b\] which implies that
 \[-ab=b(\alpha^3+\beta^3)=b(\alpha+\beta)(\alpha^2-\alpha\beta+\beta^2)=c(c^2-2b-b)=c^3-3bc\:.\]
 
 Conversely, if $R(x)$ is reducible over $\K$ then, since $\text{disc}(R(x))=-27\Delta$ is not a square in $\K$, we can choose 
 $\alpha$ and $\beta$ in the way that $\alpha\beta(\alpha+\beta)=c\in\K$.
 We recall that we observed in the previous section that the three roots of $R(x)$ are $\alpha\beta(\alpha+\beta),\alpha\beta(\alpha\zeta_3+\beta\zeta_3^2)$ and $\alpha\beta(\alpha\zeta_3^2+\beta\zeta_3)$.
 
 In other words, one has $\alpha+\beta=\frac{c}{b}\sqrt[3]{b^2}$. Since
 $\beta=-\frac{1}{b}\cdot\alpha^5\sqrt[3]b-\frac{a}{b}\cdot\alpha^2\sqrt[3]b$ we get 
 \begin{equation}
  \label{combination}
  b\alpha-\alpha^5\sqrt[3]b-a\alpha^2\sqrt[3]b-c\sqrt[3]{b^2}=0
 \end{equation}
 which is a non-trivial $\K$-linear combination between elements of \[\mathcal C=\left\{\alpha^i\sqrt[3]{b^j}\:|\:i=0,...,5;\:j=0,1,2\right\}\:.\] Suppose that $[\F:\K]=36$ then $[\K(\sqrt[3]b,\alpha):\K]=18$ and $\mathcal C$ is a basis of $\K(\sqrt[3]b,\alpha)$; this is a contradiction because of Equation \eqref{combination}.
\end{proof}
In particular, it follows that, under the hypothesis of Proposition \ref{prop:D_6}, we can choose $\alpha$ and $\beta$ in the way that $\alpha+\beta=\frac{c}{b}\sqrt[3]{b^2}$, where $c$ is the rational root of $R(x)$. Moreover, we have that $\K(\sqrt[3]b)$ is the only cubic subfield of $\K(\alpha)=\K(\beta)$.
\begin{cor}
 Suppose that $\cchar(\F)\neq2$, the field $\F$ does not contain all the cubic roots of unity and $\Delta\neq-3n^2$. 
 
 We have that $\Gal{\F}{\K}\cong\mathcal S_3\times\mathcal S_3$ if and only if $\sqrt[3]b\notin\K$ and $R(x)$ is irreducible over $\K$.
\end{cor}
\begin{proof}
 It follows immediately from Lemma \ref{lemma:both}, Propositions \ref{prop:cube} and \ref{prop:D_6}.
\end{proof}
We now give some examples of polynomials for which we can compute the Galois group.
\begin{cor}
 \label{cor:1}
 We consider infinite families of $a,b\in\Z$ when $f(x)=x^6+ax^3+b\in\Q[x]$ is irreducible over $\Q$ and $\Delta\neq-3n^2$. Hence, in the following cases one has $\Gal{\F}{\Q}\cong\mathcal S_3\times\mathcal S_3$:
 \begin{enumerate}
  \item $b=2$ and $a\equiv1\emph{ mod }10$;
  \item $b=27(100n+23)^3+1$ with $n\leq-1$ and $a\equiv25\emph{ mod }30$.
 \end{enumerate}
 While in the following ones $\Gal{\F}{\Q}\cong D_6$:
 \begin{enumerate}
  \setcounter{enumi}{2}
  \item $b\equiv10\emph{ mod }12$ and $a=0$;
  \item $b=-(5n+1)^3$ with $n>0$ and $a\equiv3\emph{ mod 5}$.
 \end{enumerate}
\end{cor}
\begin{proof}
 Cases 1 and 2 follows applying Eisenstein criterion to $R(x)$, while in Case 3 we observe that $0$ is a root of $R(x)$. Finally, for Case 4 we note that $\sqrt[3]b=-5n-1\in\Z$.
\end{proof}
From now on we suppose that $\Delta=-3n^2$ for some $n\in\K$. This implies $\K(\zeta_3)\subset\K(\alpha)$ for Lemma \ref{lemma:zeta}.
From what we said in the previous section, see Figure \ref{Splitting}, we have that
\[[\F:\K]=\left\{
 \begin{aligned}
  &18\hspace{1cm}\Longrightarrow\hspace{1cm}\Gal{\F}{\K}\cong      
   C_3\times\mathcal S_3 \\
  &6\hspace{1.2cm}\Longrightarrow\hspace{1cm}\Gal{\F}{\K}\cong\mathcal S_3\text{ or }C_6\:.
 \end{aligned}
\right. \] 
All these three cases can happen, here we give an example for each group over the rational numbers:
\begin{itemize}
    \item $C_3\times\mathcal S_3$: \hspace{1cm}\:\:$x^6+3x^3+3$;
    \item $\mathcal S_3$: \hspace{2cm}\:$x^6+3$;
    \item $C_6$: \hspace{2cm}$x^6+x^3+1$.
\end{itemize}
\begin{prop}
 \label{prop:C_6}
 Suppose that $\K$ is a field without all the cubic roots of unity and such that $\cchar(\K)\neq2$. Moreover, we assume $\Delta=-3n^2$.
 
 Then one has $\Gal{\F}{\K}\cong C_6$ if and only if $\sqrt[3]b\in\K$.
\end{prop}
\begin{proof}
 We start with the if implication. Suppose $\sqrt[3]b=d\in\K$. This implies that $d(\alpha+\beta)$ is a root of $R(x)=x^3-3bx+ab$, for the right choice of $\alpha$ and $\beta$. We have that
 \[\text{disc}(R(x))=108b^3-27a^2b^2=-27b^2\Delta\] and then
 \begin{equation}
 \label{R}
 \sqrt{\text{disc}(R(x))}=3b\cdot i\sqrt3\cdot\sqrt\Delta=-9bn\in\K\:.
 \end{equation}
 From this and the fact that $R(x)$ is irreducible for Lemma \ref{lemma:both} we obtain that $\Gal{\M}{\K}\cong C_3$, where $\M$ is the splitting field of $R(x)$, and $\M=\K(\alpha+\beta)$ is a normal extension of $\K$ of degree $3$.
 
 Since $\sqrt[3]b\in\K$ we know that $[\F:\K]=6$ and if the Galois group is not $C_6$ then it should be $\mathcal S_3$, but $\mathcal S_3$ does not have a normal subgroup of order two.
 
 Conversely, if $\sqrt[3]b\notin\K$ then $\K(\sqrt[3]b)$ and $\K(\sqrt[3]b\:\zeta_3)$ are two distinct cubic extensions of $\K$, but the group $C_6$ has only one subgroup of order two.
\end{proof}
We now prove a similar result for the case when the Galois group is $\mathcal S_3$.
\begin{prop}
 \label{prop:S_3}
 Under the same conditions in Proposition \ref{prop:C_6} we have that $\Gal{\F}{\K}\cong\mathcal S_3$ if and only if $R(x)=x^3-3bx+ab$ is reducible over $\K$. 
\end{prop}
\begin{proof}
 As before, let us begin with the if implication. Say $\alpha\beta(\alpha+\beta)=c\in\K$; then one has 
 \[b\alpha-\alpha^5\sqrt[3]b-a\alpha^2\sqrt[3]b-c\sqrt[3]{b^2}=0\]
 and, like in the proof of Proposition \ref{prop:D_6}, this is a non-trivial $\K$-linear combination between elements of the basis
 \[\mathcal D=\left\{\alpha^i\sqrt[3]{b^j}\:|\:i=0,...,5;\:j=0,1,2\right\}\] of $\K(\sqrt[3]b\:,\alpha)$.
 Hence, the Galois group cannot have order $18$.
 Clearly, we also have that $\Gal{\F}{\K}$ is not isomorphic to $C_6$ for Proposition \ref{prop:C_6} and we conclude using Lemma \ref{lemma:both}.
 
 Conversely, suppose that $\Gal{\F}{\K}\cong\mathcal S_3$. Then we know that $\sqrt[3]b\notin\K$ for Proposition \ref{prop:C_6} and $\F=\K(\sqrt[3]b\:,\zeta_3)$. If we assume that $R(x)$ is irreducible then we can prove that its splitting field is a normal extension of $\K$ of degree 3, using Equation \eqref{R}.
 This is a contradiction because $\mathcal S_3$ doesn not have normal subgroups of order two.
\end{proof}
We conclude with the following corollary.
\begin{cor}
 Under the same conditions in Proposition \ref{prop:C_6} we have that $\Gal{\F}{\K}\cong C_3\times\mathcal S_3$ if and only if $\sqrt[3]b\notin\K$ and $R(x)$ is irreducible over $\K$.
\end{cor}
\begin{proof}
 It follows immediately from Lemma \ref{lemma:both}, Propositions \ref{prop:C_6} and \ref{prop:S_3}.
\end{proof}
Furthermore, we produce some families of polynomials whose Galois groups can be determined from the results in this subsection.
\begin{cor}
 \label{cor:2}
 Suppose that $\K=\Q$. Then we have that
 \begin{enumerate}
  \item if $f(x)=x^6+3(3n+1)^2$ then   
        $\Gal{\F}{\Q}\cong\mathcal S_3$;
  \item if $f(x)=x^6+(5n+1)^3x^3+(5n+1)^6$ then $\Gal{\F}{\Q}\cong C_6$;       
  \item if $f(x)=x^6+px^3+p^2$ where $p\equiv1\emph{ mod }5$ is prime then $\Gal{\F}{\Q}\cong C_3\times\mathcal S_3$.
 \end{enumerate}
\end{cor}
\begin{proof}
 In Case 1 the polynomial $R(x)$ is clearly reducible and $3$ is not a square in $\Q$. In Case 2 $b$ is a cube and in Case 3 we see that $R(x)=x^3-3p^2x+p^3$. 
 This polynomial is irreducible because the only possible rational solutions are $\{1,p,p^2,p^3\}$, but it is easy to check that none of these actually makes $R(x)$ vanish.
\end{proof}

\subsection{Proof of Theorem \ref{teo:main}: cubic roots of unity in \texorpdfstring{$\K$}{K}}
In this subsection we suppose that $\K$ is a field with characteristic different from two, but such that all the three cubic roots of unity belong to $\K$. This means that either $\cchar(\K)=3$ or $\zeta_3\in\K$.

Note that if $f(x)=x^6+ax^3+b$ is irreducible over $\K$ then  $\Delta\neq-3n^2$ for every $\in\K$.
\begin{prop}
 \label{prop:zeta_in}
 Suppose that $\zeta_3\in\K$ and $f(x)=x^6+ax^3+b$ is irreducible over $\K$, a field such that $\cchar(\K)\neq2$. Denote with $\F$ the splitting field of $f(x)$ as before. Then we have that
 \begin{itemize}
     \item $\Gal{\F}{\K}\cong\mathcal S_3$ if and only if  
           $\sqrt[3]b\in\K$ and $R(x)=x^3-3bx+ab$ is irreducible over $\K$;
     \item $\Gal{\F}{\K}\cong C_6$ if and only if $\sqrt[3]b\notin\K$ and $R(x)$ is reducible over $\K$.
 \end{itemize}
\end{prop}
\begin{proof}
 Let us prove the if implications first. Since $\sqrt[3]b\:,\zeta_3\in\K$ one has $[\F:\K]=6$ and $\F=\K(\sqrt\Delta,\alpha+\beta)$ because $\sqrt[3]b\:(\alpha+\beta)$ is a root of $R(x)$ which is irreducible.
 Moreover, the field $\F$ is the splitting field of $R(x)$ because
 $\cchar(\K)\neq3$ and $\sqrt{\text{disc}(R(x))}=m\sqrt\Delta$ for some non-zero $m\in\K$. Then it follows from a standard result in Galois theory, see \cite{Cox}, that $\Gal{\F}{\K}$ is isomorphic to $\mathcal S_3$. Note that here we use that $\cchar(\K)\neq2$.
 
 Now suppose that $\sqrt[3]b\notin\K$ and $R(x)$ is reducible. The fact that $[\F:\K]\neq18$ follows in the same way as in the proof of Propositions \ref{prop:D_6} and \ref{prop:S_3}. Then we have that $\F=\K(\sqrt\Delta,\sqrt[3]b)$, but $\K(\sqrt[3]b)$ is a normal extension of $\K$ because it is the splitting field of $x^3-b$. This implies that $\Gal{\F}{\K}\cong C_6$.
 
 Conversely, suppose that $\Gal{\F}{\K}$ is isomorphic to $C_6$. We immediately obtain that $\sqrt[3]b\notin\K$ for what we said before and Lemma \ref{lemma:both}. If $R(x)$ is irreducible then, since $\text{disc}(R(x))$ is not a square in $\K$ and $\cchar(\K)\neq2$, the Galois group should be $\mathcal S_3$, but this is a contradiction.
 
 Finally, we say that $\Gal{\F}{\K}\cong\mathcal S_3$. Then $R(x)$ is irreducible over $\K$, again for what we said before and Lemma \ref{lemma:both}, and $\K(\sqrt[3]b)$ is a normal extension of degree three of $\K$ if we suppose that $\sqrt[3]b$ is not in $\K$. This completes the proof because $\mathcal S_3$ does not have a normal subgroup of order two.
\end{proof}
From the previous proposition we also obtain the following result.
\begin{cor}
 With the hypothesis of Proposition \ref{prop:zeta_in} we have that $\Gal{\F}{\K}\cong C_3\times\mathcal S_3$ if and only if $\sqrt[3]b\notin\K$ and $R(x)$ is irreducible over $\K$.
\end{cor}
\begin{proof}
 It follows immediately from Lemma \ref{lemma:both} and Proposition \ref{prop:zeta_in}.
\end{proof}
The following corollary shows another family of polynomials with Galois group isomorphic to $C_6$.
\begin{cor}
 If $f(x)=x^6-p$ where $p$ is a prime integer then $\emph{Gal}\left(\faktor{\F}{\Q(\zeta_3)}\right)\cong C_6$.
\end{cor}
\begin{proof}
 We note that $f(x)$ is always irreducible over $\Q(\zeta_3)$. In fact, it is irreducible over $\Q$ for the Eisenstein criterion and one has $[\Q(\zeta_3,\sqrt[6]p):\Q]=12$. Moreover, one has $\sqrt[3]p\notin\Q(\zeta_3)$ since $[\Q(\zeta_3):\Q]=2$, while $R(x)=x^3+3px$ is clearly reducible.
\end{proof}
There is only one case which is left to study: when $\K$ is a field with characteristic equal to three. We do this in the following proposition.
\begin{prop}
 Suppose that $\K$ is a field such that $\cchar(\K)=3$ and $f(x)=x^6+ax^3+b$ is irreducible over $\K$.
 
 Then we have that $[\F:\K]=6,18$ depending on whether $\K(\sqrt[3]a\:,\sqrt[3]b)$ is degree $3$ or $9$ over $\K$, but $f(x)$ and $\F$ are not separable and then we cannot define the Galois group. 
\end{prop}
\begin{proof}
 Since $\cchar(\K)=3$ we have that $R(x)=x^3+ab$ and $\alpha+\beta=-\sqrt[3]a$. We observe that $f(x)$ factors as \begin{equation}
  \label{factor}
  f(x)=\left(x^2+\sqrt[3]a\:x+\sqrt[3]b\right)^3 
 \end{equation} in the algebraic closure of $\K$. Therefore, the field $\F$ coincides with the splitting field of $x^2+\sqrt[3]a\:x+\sqrt[3]b$ over $\K(\sqrt[3]a\:,\sqrt[3]b)$.
 
 We have that $\K(\sqrt[3]a)$ and $\K(\sqrt[3]b)$ are two extensions of $\K$ of degree one or three. Furthermore, if $\sqrt[3]a\:,\sqrt[3]b\in\K$ then the factor $x^2+\sqrt[3]a\:x+\sqrt[3]b$ of $f(x)$ in Equation \eqref{factor} belongs to $\K[x]$, but $f(x)$ is irreducible over $\K$. In other words, one has $[\K(\sqrt[3]a\:,\sqrt[3]b):\K]=3,9$; in fact, if the degree was 6 then $\sqrt[3]b$ should be the root of a degree 2 polynomial in $\K(\sqrt[3]a)$, and  \[[\F:\K]=[\F:\K(\sqrt[3]a\:,\sqrt[3]b)]\cdot[\K(\sqrt[3]a\:,\sqrt[3]b):\K]=6,18\:.\] 
 Since $f(x)$ is irreducible we have that it is the minimal polynomial of $\alpha$ over $\K$. The roots of $f(x)$ are $\alpha$ and $\beta$, each one with multiplicity three, which means $\alpha$ is not separable over $\K$. In particular, this implies that $\F$ is not a separable extension of $\K$ and then there is no Galois group in this case. 
\end{proof}
Note that irreducible polynomials of this kind really exist: examples are given by $x^6-u$ and $x^6+uvx^3+u$ in $\F_3(u,v)[x]$, where $u$ and $v$ are transcendental over $\F_3$. The proposition we proved before clearly implies that their splitting field $\F$ is not a separable extension.
Moreover, we can also observe that the group of automorphisms of $\F$ only has two elements, despite $[\F:\F_3(u,v)]$ being higher.

\subsection{Irreducible polynomials over finite fields}
In this subsection $\K$ is a finite field; in other words, we consider $\K=\F_{p^k}$ where $p$ is an odd prime and $k\geq1$ is an integer.
We say that $p\neq2$ because we want $\cchar(\K)$ different from two.

It is known from Galois theory \cite{Cox} that the algebraic extensions of a finite field are uniquely determined from the degree and are always separable. Moreover, it follows immediately from this fact that  Galois groups of finite extensions of $\F_{p^k}$ are cyclic. 

If we take an irreducible polynomial $f(x)=x^6+ax^3+b$ in $\F_{p^k}[x]$ then it has to be $\Gal{\F}{\F_{p^k}}\cong C_6$, where $\F$ is the splitting field of $f(x)$.
Since the results obtained in the previous subsections hold for finite fields too, provided that the characteristic is not two, we can prove a criterion that allows us to say when a polynomial $f(x)$ as before is irreducible.
\begin{proof}[Proof of Theorem \ref{teo:finite}]
 Let us consider Case 1. If $f(x)$ is irreducible then the Galois group of the splitting field $\F$ is isomorphic to $C_6$. The claim follows from Proposition \ref{prop:C_6} because $\zeta_3$ is not in the field for our assumption on $p^k$.
 
 Suppose that $R(x)$ is irreducible. Then the splitting field of $R(x)$ over $\F_{p^k}$ has degree three. This means that, since $\cchar(\K)\neq2$, the discriminant of $R(x)$ is a square in $\F_{p^k}$; in other words, we have that $-3\Delta$ is a square. 
 
 Now, we use the hypothesis that $p^k\not\equiv1\:\text{mod }3$ and $p$ is not three, which tells us that $\zeta_3\notin\F_{p^k}$ and then $-3$ and $\Delta$ are both not squares in $\F_{p^k}$. This implies that $\F$ contains an extension of $\F_{p^k}$ of degree two and one of degree three; hence, one has $[\F:\F_{p^k}]=6$.
 
 We conclude that $f(x)$ is irreducible because the only factorizations of $f(x)$ in the algebraic closure of $\F_{p^k}$ are: products of two degree three polynomials or products of three degree two polynomials.
 In these cases one has $[\F,\F_{p^k}]=3,2$. We cannot have $\alpha\in\F_{p^k}$ because otherwise $2\alpha^3+b=\sqrt\Delta$ also belongs to $\F_{p^k}$; this is impossible for what we said before.
 
 We now prove Case 2. If $f(x)$ is irreducible then the claim follows in the same way as in Case 1 from Proposition \ref{prop:zeta_in} because 
 $p^k\equiv1\:\text{mod }3$ implies $\zeta_3\in\F_{p^k}$.
 Suppose that $\sqrt[3]b\:,\sqrt\Delta\notin\F_{p^k}$; then we have that $[\F:\F_{p^k}]=6$ and we conclude as in Case 1. 
 
 Finally, if $p=3$ then we have that the map $x\rightarrow x^3$ is the Frobenius automorphism of $\F_{3^k}$. This means that any element in the field is a cube and 
 $f(x)=\left(x^2+\sqrt[3]a\:x+\sqrt[3]b\right)^3$ is reducible over $\F_{p^k}$. 
\end{proof}
We note that $p^2$ is always equal to $1\:\text{mod }3$ when $p\neq3$ and then Case 1 can only happen if $k$ is odd. 

\section{Subfields}
\label{section:subfields}
In this section we use Galois correspondence to give a complete list of all the intermediate extensions of $\K$ inside the splitting field $\F$ of $f(x)=x^6+ax^3+b$. We consider all the possible cases that we studied before in the paper. Moreover, we assume $\cchar(\K)\neq2,3$.

\subsection{\texorpdfstring{$\zeta_3\notin\K$}{Zeta3 not in K} and \texorpdfstring{$\Delta\neq-3n^2$}{Delta not -3n2}}
\paragraph{\texorpdfstring{$\sqrt[3]b\notin\K$}{Sqrt3 b not in K} and \texorpdfstring{$R(x)$}{R(x)} reducible over \texorpdfstring{$\K$}{K}} We explicitly described the Galois group in Subsection \ref{subsection:not}. We have that $\Gal{\F}{\K}\cong D_6$ and there are $14$ proper subfields, as shown in Figure \ref{Subgroups_D6}.
\begin{itemize}
    \item Degree 2: \[\K(\sqrt\Delta),\:\K(\zeta_3),\:\K(\sqrt{-3\Delta})\hspace{2cm}\text{ normal}\]
          we recall that $\K(\zeta_3)=\K(i\sqrt3)$;
    \item Degree 3: \[\K(\sqrt[3]b),\:\K(\sqrt[3]b\:\zeta_3),\:\K(\sqrt[3]b\:\zeta_3^2)\]
          since when we say $\sqrt[3]b\notin\K$ we mean exactly that $x^3-b$ is irreducible;
    \item Degree 4: \[\K(\sqrt\Delta,\zeta_3)\hspace{2cm}\text{ normal}\:;\]
    \item Degree 6: \[\K(\zeta_3,\sqrt[3]b)\hspace{2cm}\text{ normal}\]
          this is the splitting field of $x^3-b$ over $\K$,
          \[\K(\alpha),\:\K(\alpha\zeta_3),\:\K(\alpha\zeta_3^2)\]
          \[\K(\alpha\zeta_3+\beta\zeta_3^2),\:\K(\alpha+\beta\zeta_3),\:\K(\alpha+\beta\zeta_3^2)\]
          each of these extensions is fixed by exactly one symmetry of $D_6$, see Figure \ref{D_6}.
\end{itemize}
\paragraph{$\sqrt[3]b\in\K$ and $R(x)$ irreducible over $\K$} The Galois group is isomorphic to $D_6$ like in the previous case. 
\begin{figure}[H]
\begin{center}
 \renewcommand{\arraystretch}{2}
  \begin{tabular}{ | c || c | c | }
    \hline
    $1$ & $1$ & $1$ \\ \hline
    $2$ & $7$ & $1$ \\ \hline
    $3$ & $1$ & $1$ \\ \hline 
    $4$ & $3$ & $0$ \\ \hline
    $6$ & $3$ & $3$ \\ \hline
    $12$ & $1$ & $1$ \\
    \hline
  \end{tabular}
 \renewcommand{\arraystretch}{1}
\end{center}
\caption{This table shows the number of subgroups of $D_6$ (central column) of a given order (left column). The right column indicates the number of normal subgroups.}
\label{Subgroups_D6}
\end{figure}
Then we again have $14$ subfields, see Figure \ref{Subgroups_D6}.
\begin{itemize}
    \item Degree 2: \[\K(\sqrt\Delta),\:\K(\zeta_3),\:\K(\sqrt{-3\Delta})\hspace{2cm}\text{ normal}\]
          in this case $\K(\sqrt{-3\Delta})$ is the quadratic extension that corresponds to $C_6\triangleleft D_6$;
    \item Degree 3: \[\K(\alpha+\beta),\:\K(\alpha\zeta_3+\beta\zeta_3^2),\:\K(\alpha\zeta_3^2+\beta\zeta_3
          )\] 
          since $R(x)$ is irreducible and its discriminant is not a square in $\K$, the cubic subfields are generated by its roots;
    \item Degree 4: \[\K(\sqrt\Delta,\zeta_3)\hspace{2cm}\text{ normal}\:;\]
    \item Degree 6: \[\K(\sqrt{-3\Delta},\alpha+\beta)\hspace{2cm}\text{ normal}\]
          this is the splitting field of $R(x)$ over $\K$,
          \[\K(\alpha),\:\K(\alpha\zeta_3),\:\K(\alpha\zeta_3^2)\]
          these extensions have degree six because they are generated by roots of $f(x)$, which is irreducible, and are distinct because $\zeta_3\notin\K$. Furthermore, they all contain $\sqrt\Delta$ and then they are the three symmetries of the $S_3$ subgroup of $D_6$ represented by $\K(\sqrt\Delta)$,
          \[\K(\zeta_3,\alpha+\beta),\:\K(\zeta_3,\alpha\zeta_3+\beta\zeta_3^2),\:\K(\zeta_3,\alpha\zeta_3^2+\beta\zeta_3)\] 
          corresponding to the three symmetries of the $S_3$ subgroup of $D_6$ represented by $\K(\zeta_3)$;
          such subgroup also corresponds to $\Gal{\F}{\K(\zeta_3)}$, which is the splitting field of $R(x)$ over $\K(\zeta_3)$.     
\end{itemize}
\paragraph{$\sqrt[3]b\notin\K$ and $R(x)$ irreducible over $\K$} The Galois group is isomorphic to $\mathcal S_3\times\mathcal S_3$. There are $58$ subfields as shown in Figure \ref{Subgroups_S3S3}.
\begin{figure}[H]
\begin{center}
\begin{subfigure}[b]{0.3\textwidth}
 \renewcommand{\arraystretch}{2}
  \begin{tabular}{ | c || c | c | }
    \hline
    $1$ & $1$ & $1$ \\ \hline
    $2$ & $15$ & $0$ \\ \hline
    $3$ & $4$ & $2$ \\ \hline 
    $4$ & $9$ & $0$ \\ \hline
    $6$ & $20$ & $2$ \\ \hline
    $9$ & $1$ & $1$ \\ \hline
    $12$ & $6$ & $0$ \\ \hline
    $18$ & $3$ & $3$ \\ \hline
    $36$ & $1$ & $1$ \\
    \hline
  \end{tabular}
 \renewcommand{\arraystretch}{1}
 \end{subfigure}
 \begin{subfigure}[c]{0.3\textwidth}
 \begin{tikzpicture}[node distance=2cm, auto]
               \node (A) {$\F$};
		       \node (B) [below of=A,xshift=2cm]   
		                 {$\K\left(\sqrt{-3\Delta},\sqrt[3]b(\alpha+\beta)\right)$};
               \node (C) [below of=A,xshift=-2cm] {$\K(\zeta_3,\sqrt[3]b)$};
               \node (D) [below of=C,xshift=2cm]          {$\K$};
               \draw[-] (C) to node  {} (A);
               \draw[-] (B) to node  {} (A);
               \draw[-] (D) to node  {$6$} (C);
               \draw[-] (D) to node  [swap]{$6$} (B);
 \end{tikzpicture}  
 \end{subfigure}
\end{center}
\caption{Table of the subgroups of $\mathcal S_3\times\mathcal S_3$. The diagram on the right shows the two          normal extensions of $\K$ of degree six which generate $\F$.}
\label{Subgroups_S3S3}
\end{figure}
A basis for the splitting field $\F$ is  
\[\mathcal E=\left\{\alpha^i\cdot\zeta_3^j\cdot\sqrt[3]{b^k}\:\big|\:i=0,...,5;\:j=1,2;\:k=0,1,2\right\}\] and then 
$\Gal{\F}{\K}$ is completely determined by:
\[\begin{aligned}
&F(\alpha)=\alpha\zeta_3^i\:\hspace{5.1cm}i=0,1,2;\\
&F(\alpha)=\beta\zeta_3^i\:\hspace{5.1cm}i=3,4,5;\\
&F(\zeta_3)=\zeta_3^j\:\:\:\:\:\hspace{5cm}j=1,2\:;\\
&F(\sqrt[3]b)=\sqrt[3]b\:\zeta_3^k\hspace{4.5cm}k=0,1,2\:.
\end{aligned}\]
It is possible to distinguish each pair of the following subfields, say $\K_1$ and $\K_2$ for example, by showing that one of these automorphisms fixes $\K_1$, but not $\K_2$. 
\begin{itemize}
    \item Degree 2: \[\K(\zeta_3),\:\K(\sqrt{-3\Delta}),\:\K(\sqrt\Delta)\hspace{2cm}\text{ normal}\]
          there are three subgroups of order $18$ in $\mathcal S_3\times\mathcal S_3$: two are isomorphic to $C_3\times\mathcal S_3$, while the third one to $C_3\rtimes\mathcal S_3$. The latter subgroup corresponds to $\K(\sqrt\Delta)$;
    \item Degree 3: \[\K(\sqrt[3]b),\:\K(\sqrt[3]b\:\zeta_3),\:\K(\sqrt[3]b\:\zeta_3^2)\:,\]    
          \[\K\left(\sqrt[3]b(\alpha+\beta)\right),\:\K\left(\sqrt[3]b(\alpha\zeta_3+\beta\zeta_3^2)\right),\:\K\left(\sqrt[3]b(\alpha\zeta_3^2+\beta\zeta_3)\right)\:;\]
    \item Degree 4: \[\K(\sqrt\Delta,\zeta_3)\hspace{2cm}\text{ normal}\:;\]
    \item Degree 6: \[\K(\zeta_3,\sqrt[3]b),\:\K\left(\sqrt{-3\Delta},\sqrt[3]b(\alpha+\beta)\right)\hspace       {2cm}\text{ normal}\]
          see Figure \ref{Subgroups_S3S3},
          \[\K\left(\zeta_3,\sqrt[3]b(\alpha+\beta)\right),\:\K\left(\zeta_3,\sqrt[3]b(\alpha\zeta_3+\beta\zeta_3^2)\right),\:\K\left(\zeta_3,\sqrt[3]b(\alpha\zeta_3^2+\beta\zeta_3)\right)\]
          these three extensions are generated by $\K(\zeta_3)$ and the three cubic subfields of $\K\left(\sqrt{-3\Delta},\sqrt[3]b(\alpha+\beta)\right)$,
          \[\K(\sqrt{-3\Delta},\sqrt[3]b),\:\K(\sqrt{-3\Delta},\sqrt[3]b\:\zeta_3),\:\K(\sqrt{-3\Delta},\sqrt[3]b\:\zeta_3^2)\]
          these other three are instead generated by $\K(\sqrt{-3\Delta})$ and the three cubic subfields of $\K(\zeta_3,\sqrt[3]b)$,
          \[\K(\sqrt\Delta,\sqrt[3]b),\:\K(\sqrt\Delta,\sqrt[3]b\:\zeta_3),\:\K(\sqrt\Delta,\sqrt[3]b\:\zeta_3^2)\]
          they are generated by $\K(\sqrt{\Delta})$ and the three cubic subfields of $\K(\zeta_3,\sqrt[3]b)$,
          \[\K\left(\sqrt\Delta,\sqrt[3]b(\alpha+\beta)\right),\:\K\left(\sqrt\Delta,\sqrt[3]b(\alpha\zeta_3+\beta\zeta_3^2)\right),\:\K\left(\sqrt\Delta,\sqrt[3]b(\alpha\zeta_3^2+\beta\zeta_3)\right)\]
          as before these extensions are generated by $\K(\sqrt\Delta)$ and the three cubic subfields of $\K\left(\sqrt{-3\Delta},\sqrt[3]b(\alpha+\beta)\right)$,
          \[\K(\alpha),\:\K(\alpha\zeta_3),\:\K(\alpha\zeta_3^2),\:\K(\beta),\:\K(\beta\zeta_3),\:\K(\beta\zeta_3^2)\]
          these subfields have the property of not containing any cubic extension of $\K$, because they correspond to the six subgroups of $\mathcal S_3\times\mathcal S_3$ of order six which are only contained by the $C_3\rtimes\mathcal S_3$ subgroup;
    \item Degree 9: \[\K(\alpha+\beta),\:\K(\alpha\zeta_3+\beta\zeta_3^2),\:\K(\alpha\zeta_3^2+\beta\zeta_3       ),\:\K((\alpha+\beta)\zeta_3^2),\:\K(\alpha+\beta\zeta_3),\:\K(\alpha\zeta_3+\beta)\:,\]
          \[\K((\alpha+\beta)\zeta_3),\:\K(\alpha\zeta_3^2+\beta),\:\K(\alpha+\beta\zeta_3^2)\]
          each of the $9$ subgroups of $\mathcal S_3\times\mathcal S_3$ of order four is only contained in two subgroups of order $12$, which are maximal. Hence, the extensions of degree $9$ are generated by the elements obtained by taking the quotient of the generators of every cubic subfield of $\F$;
    \item Degree 12: \[\K\left(\sqrt\Delta,\zeta_3,\sqrt[3]b\right),\:\K\left(\sqrt\Delta,\zeta_3,\sqrt[3]b       (\alpha+\beta)\right)\hspace{2cm}\text{ normal}\]
          these are the splitting fields of $(x^3-b)(x^2-\Delta)$ and $R(x)(x^2-\Delta)$,
          \[\K(\zeta_3,\alpha),\:\K(\zeta_3,\beta)\]
          these extensions are clearly of degree $12$ and they are distinct because $\beta=\frac{\sqrt[3]b}{\alpha}\notin\K(\zeta_3,\alpha)$; moreover, they are not normal because they do not contain all the roots of $f(x)$;
    \item Degree 18: \[\K(\alpha,\beta),\:\K(\alpha,\beta\zeta_3),\:\K(\alpha,\beta\zeta_3^2),\:\K(\alpha
          \zeta_3,\beta),\:\K(\alpha\zeta_3,\beta\zeta_3),\:\K(\alpha\zeta_3,\beta\zeta_3^2),\]
          \[\K(\alpha\zeta_3^2,\beta),\:\K(\alpha\zeta_3^2,\beta\zeta_3),\:\K(\alpha\zeta_3^2,\beta\zeta_3^2),\:\K(\zeta_3,\alpha+\beta),\:\K(\zeta_3,\alpha+\beta\zeta_3),\:\K(\zeta_3,\alpha+\beta\zeta_3^2),\]
          \[\K\big((\zeta_3-1)(\alpha-\beta)\big),\:\K\big(\zeta_3(\zeta_3-1)(\alpha-\beta)\big),\:\K\big(\zeta_3^2(\zeta_3-1)(\alpha-\beta)\big)\]
          the easiest way to show that these extensions are all different is to check that each one is fixed exactly by one element of order two of $\mathcal S_3\times\mathcal S_3$.
\end{itemize}

\subsection{\texorpdfstring{$\zeta_3\notin\K$}{Zeta3 not in K} and \texorpdfstring{$\Delta=-3n^2$}{Delta=-3n2}}
\paragraph{\texorpdfstring{$\sqrt[3]b\notin\K$}{Sqrt3(b) not in K} and \texorpdfstring{$R(x)$}{R(x)} reducible over \texorpdfstring{$\K$}{K}} The Galois group is isomorphic to $\mathcal S_3$ and we immediately see that the proper subfields are the following four extensions of $\K$.
\begin{itemize}
    \item Degree 2: \[\K(i\sqrt3)\hspace{2cm}\text{ normal}\:;\]
    \item Degree 3: \[\K(\sqrt[3]b),\:\K(\sqrt[3]b\:\zeta_3),\:\K(\sqrt[3]b\:\zeta_3^2)\:;\]
\end{itemize}
\paragraph{$\sqrt[3]b\in\K$ and $R(x)$ irreducible over $\K$} We have that $\Gal{\F}{\K}\cong C_6$ and then there are only two subfields of $\F$.
\begin{itemize}
    \item Degree 2: \[\K(i\sqrt3)\hspace{2cm}\text{ normal}\:;\]
    \item Degree 3: \[\K(\alpha+\beta)\hspace{2cm}\text{ normal}\:;\]
\end{itemize}
\paragraph{$\sqrt[3]b\notin\K$ and $R(x)$ irreducible over $\K$} In this case the Galois group is isomorphic to $C_3\times\mathcal S_3$ and there are $12$ subfields, see Figure \ref{Subgroups_C3S3}.
\begin{itemize}
    \item Degree 2: \[\K(i\sqrt3)\hspace{2cm}\text{ normal}\:;\]
    \item Degree 3: \[\K\left(\sqrt[3]b(\alpha+\beta)\right)\hspace{2cm}\text{ normal}\]
          this is the splitting field of $R(x)$, whose discriminant is now a square in $\K$,
          \[\K(\sqrt[3]b),\:\K(\sqrt[3]b\:\zeta_3),\:\K(\sqrt[3]b\:\zeta_3^2)\:;\]
    \item Degree 6: \[\K(\zeta_3,\sqrt[3]b),\:\K\left(\zeta_3,\sqrt[3]b(\alpha+\beta)\right)\hspace{2cm}
          \text{ normal}\]
          these extensions are the splitting fields of $x^3-b$ and $R(x)(x^2+x+1)$ respectively,
          \[\K(\alpha),\:\K(\beta)\:;\]
          \begin{figure}[H]
\begin{center}
\begin{subfigure}[b]{0.3\textwidth}
 \renewcommand{\arraystretch}{2}
  \begin{tabular}{ | c || c | c | }
    \hline
    $1$ & $1$ & $1$ \\ \hline
    $2$ & $3$ & $0$ \\ \hline
    $3$ & $4$ & $2$ \\ \hline 
    $6$ & $4$ & $1$ \\ \hline
    $9$ & $1$ & $1$ \\ \hline
    $18$ & $1$ & $1$ \\ 
    \hline
  \end{tabular}
 \renewcommand{\arraystretch}{1}
 \end{subfigure}
 \begin{subfigure}[c]{0.3\textwidth}
 \begin{tikzpicture}[node distance=2cm, auto]
               \node (A) {$\F$};
		       \node (B) [below of=A,xshift=2cm]   
		                 {$\K\left(\sqrt[3]b(\alpha+\beta)\right)$};
               \node (C) [below of=A,xshift=-2cm] {$\K(\zeta_3,\sqrt[3]b)$};
               \node (D) [below of=C,xshift=2cm]          {$\K$};
               \draw[-] (C) to node  {} (A);
               \draw[-] (B) to node  {} (A);
               \draw[-] (D) to node  {$6$} (C);
               \draw[-] (D) to node  [swap]{$3$} (B);
 \end{tikzpicture}  
 \end{subfigure}
\end{center}
\caption{Table of the subgroups of $C_3\times\mathcal S_3$. The diagram on the right shows the decomposition of $\F$ into two normal extension of $\K$, one of degree three and the other of degree six.}
\label{Subgroups_C3S3}
\end{figure}
    \item Degree 9: \[\K(\alpha+\beta),\:\K((\alpha+\beta)\zeta_3),\:\K((\alpha+\beta)\zeta_3^2)\]
          because each of three subgroup of $C_3\times\mathcal S_3$ of order two is the intersection of one non-normal order $6$ subgroup with the normal order $6$ subgroup. 
\end{itemize}

\subsection{\texorpdfstring{$\zeta_3\in\K$}{Zeta3 in K}}
\paragraph{\texorpdfstring{$\sqrt[3]b\notin\K$}{Sqrt3(b) not in K} and \texorpdfstring{$R(x)$}{R(x)} reducible over \texorpdfstring{$\K$}{K}} The Galois is isomorphic to $C_6$ and we have two subfields.
\begin{itemize}
    \item Degree 2: \[\K(\sqrt\Delta)\hspace{2cm}\text{ normal}\:;\]
    \item Degree 3: \[\K(\sqrt[3]b)\hspace{2cm}\text{ normal}\:.\]
\end{itemize}
\paragraph{$\sqrt[3]b\in\K$ and $R(x)$ irreducible over $\K$} The Galois is isomorphic to $\mathcal S_3$ and we have four subfields.
\begin{itemize}
    \item Degree 2: \[\K(\sqrt\Delta)\hspace{2cm}\text{ normal}\:;\]
    \item Degree 3: \[\K(\alpha+\beta),\:\K(\alpha\zeta_3+\beta\zeta_3^2),\:\K(\alpha\zeta_3^2+\beta\zeta_3       )\]
          generated by the three roots of $R(x)$.
\end{itemize}
\paragraph{$\sqrt[3]b\notin\K$ and $R(x)$ irreducible over $\K$} We have that $\Gal{\F}{\K}\cong C_3\times\mathcal S_3$ and we have $12$ subfields, see Figures \ref{Subgroups_C3S3} and \ref{Diagram_C3S3}.
\begin{figure}[H]
\begin{center}
 \begin{tikzpicture}[node distance=2cm, auto]
               \node (A) {$\F$};
		       \node (B) [below of=A,xshift=2cm]   
		                 {$\K\left(\sqrt\Delta,\sqrt[3]b(\alpha+\beta)\right)$};
               \node (C) [below of=A,xshift=-2cm] {$\K(\sqrt[3]b)$};
               \node (D) [below of=C,xshift=2cm]          {$\K$};
               \draw[-] (C) to node  {} (A);
               \draw[-] (B) to node  {} (A);
               \draw[-] (D) to node  {$3$} (C);
               \draw[-] (D) to node  [swap]{$6$} (B);
 \end{tikzpicture}  
\end{center}
\caption{The diagram shows the decomposition of $\F$ into two normal extension of $\K$, one of degree three and the other of degree six.}
\label{Diagram_C3S3}
\end{figure}
\begin{itemize}
    \item Degree 2: \[\K(\sqrt\Delta)\hspace{2cm}\text{ normal}\:;\]
    \item Degree 3: \[\K(\sqrt[3]b)\hspace{2cm}\text{ normal}\]
          this is the splitting field of $x^3-b$,
          \[\K\left(\sqrt[3]b(\alpha+\beta)\right),\:\K\left(\sqrt[3]b(\alpha\zeta_3+\beta\zeta_3^2)\right),\:\K\left(\sqrt[3]b(\alpha\zeta_3^2+\beta\zeta_3)\right)\]
          the three roots of $R(x)$;
    \item Degree 6: \[\K\left(\sqrt\Delta,\sqrt[3]b(\alpha+\beta)\right),\:\K\big(\sqrt\Delta,\sqrt[3]b\big       )\hspace{2cm}\text{ normal}\]
          these extensions are the splitting fields of $R(x)$ and $(x^3-b)(x^2-\Delta)$ respectively,
          \[\K(\alpha),\:\K(\beta)\:;\]
    \item Degree 9: \[\K\left(\alpha+\beta\right),\:\K\left(\alpha\zeta_3+\beta\zeta_3^2\right),\:\K\left(
          \alpha\zeta_3^2+\beta\zeta_3\right)\]
          again each of three subgroup of $C_3\times\mathcal S_3$ of order two is the intersection of one non-normal order $6$ subgroup with the normal order $6$ subgroup.
\end{itemize}

\end{document}